\documentclass[12pt, leqno, letterpaper]{article}

\usepackage{amsmath,amssymb, amsthm}
\usepackage{xcolor}
\usepackage{url}
\usepackage{subfig}
\usepackage[english]{babel}
\usepackage{graphicx}
\graphicspath{{code/}}
\newcommand{\comment}[1]{}
\newtheorem{theorem}{Theorem}[section]

\newtheorem{lemma}{Lemma}[section]

\newcommand{\footremember}[2]{%
    \footnote{#2}
    \newcounter{#1}
    \setcounter{#1}{\value{footnote}}%
}

\theoremstyle{remark}

\theoremstyle{definition}

\providecommand{\keywords}[1]{\textbf{\textit{key words:}} #1}

\begin{document}

\title{Large and moderate deviations for a discrete-time marked Hawkes process}
\author{Haixu Wang \footremember{hw}{Department of Mathematics, Florida State University, Tallahassee, FL 32306 \newline Email: hwang@math.fsu.edu}}
\date{}
\maketitle
\begin{abstract}
    Hawkes process is continuous-time stochastic model that captures temporal stochastic self-exciting phenomena. In particular, the linear Hawkes process has been well studied and widely used in practice because of mathematical tractability. However, in some contexts, a Hawkes model is not applicable because data is recorded in a discrete-time scheme or an aggregated way. Thus, a discrete-time Hawkes model is appealing for applications.
	In this paper, we study large and moderate deviations for a discrete-time marked Hawkes process first proposed in \cite{XZW2020}.
\end{abstract}
\small \keywords{discrete-time; marked Hawkes process; self-exciting; univariate; large deviations; moderate deviations}

\section{Introduction}
\label{sec:introduction}
Hawkes process is a self-exciting simple point process named after \cite{Hawkes}. Hawkes processes originate from statistical literature to model the occurrences of earthquakes and shocks after earthquakes, see \cite{VereJones}. In contrast to a standard Poisson process, the intensity of Hawkes process depends on its entire history, which can model the self-exciting or clustering effect. 
In finance, most applications of Hawkes processes are about high-frequency trading \cite{BH,VM}. Furthermore, Hawkes processes have been used to model credit default and the arrival of company defaults in a bond portfolio \cite{Eg,GGD}. Recently, Hawkes models have been applied in social networks. For example,  \cite{FSCB} modelled the rate of sending email for each officer at the West Point Military Academy. The more applications of Hawkes process can be found in seismology, neuroscience, cosmology, ecology, and epidemiology. For a list of references for these applications, see \cite{Bordenave, ZhuTh, Liniger}. 

Next, let us introduce the Hawkes process. Let $N$ be a simple point process on $\mathbb{R}$ and let $$\mathcal{F}^{-\infty}_t:=\sigma\left(N(C), C\in\mathcal{B}\left(\mathbb{R}\right), C\in(-\infty,t]\right)$$ be an increasing family of $\sigma$-algebras. Any non-negative $\mathcal{F}^{-\infty}_t$-progressively measurable process $\lambda_t$ with 
$$\mathbb{E}\left[N(a,b]|\mathcal{F}^{-\infty}_a\right]=\mathbb{E}\left[\int_{a}^{b}\lambda_s ds|\mathcal{F}^{-\infty}_a\right]$$
a.s. for all interval $(a,b]$ is called the $\mathcal{F}^{-\infty}_t$-intensity of N. $N_t:=N(0,t]$ denotes the number of points in the interval $(0,t]$. In general, a marked Hawkes process with intensity defined as
\begin{equation}
	\label{gli}
	\lambda_t := \lambda\left(\int_{(-\infty,t)\times\mathbb{X}}h(t-s,\ell)N(ds,d\ell)\right),
\end{equation}
where $\lambda(\cdot):\mathbb{R}^{+}\to\mathbb{R}^{+}$ is locally integrable and left continuous, $h(\cdot,\cdot):\mathbb{R}^{+}\times\mathbb{X}\to\mathbb{R}^{+}$ is integrable, $\ell$ denotes the mark variable, and $||h||_{L^1} = \int_{0}^{\infty}\int_{\mathbb{X}}h(t,\ell)q(d\ell)dt<\infty$. Here $\mathbb{X}$ is measurable space with common law $q(d\ell)$. $h(\cdot)$ and $\lambda(\cdot)$ are referred as exciting function and rate function, respectively.
Local integrability assumption of $\lambda(\cdot)$ ensures that the process is non-explosive and left continuity assumption ensures that $\lambda_t$ is $\mathcal{F}_t$-predictable. 
The integral in equation $\eqref{gli}$ stands for $\int_{(-\infty,t)\times \mathbb{X}}h(t-s,\ell)N(ds,d\ell)=\sum_{\tau_i<t}h(t-\tau_i,\ell_i)$, where $(\tau_i)_{i\ge1}$ are the occurrences of the points before time t, and the $(\ell_i)_{i\ge1}$ are i.i.d. random marks, $\ell_i$ being independent of previous arrival times $\tau_j$, $j\le i$.

When $\lambda(\cdot)$ is linear, it is called a linear Hawkes process. There were extensive studies on the stability, law of large numbers, central limit theorems, large deviations, Bartlett spectrum, etc. 
In particular, \cite{Bacry} proved the functional law of large numbers and the functional central limit theorems. \cite{Bordenave} derived large deviations of Hawkes process. For a survey on Hawkes processes and related self-exciting processes, Poisson cluster processes, affine point processes, etc., see \cite{DV}.

When $\lambda(\cdot)$ is nonlinear, it is known as a nonlinear Hawkes process. Because of the lack of immigration-birth representation and computational tractability, nonlinear Hawkes processes are much less studied. However, there were some efforts in this direction. A nonlinear Hawkes process
was first introduced by \cite{BM}. The central limit theorems, the large deviation principles for nonlinear Hawkes processes can be found in \cite{ZhuIV,ZhuI,ZhuII,ZhuIII}. 

Hawkes process can also be extended to the multivariate setting. For a survey of multivariate processes and a short history of Hawkes process, we refer to \cite{Liniger}.

In contrast to the continuous setting, in reality, the arrivals of events are often recorded in a discrete-time scheme. For example, the data is collected on a fixed phase or the data only shows the aggregate results. Continuous-time Hawkes processes can model the unevenly spaced the arrival of events in time, while modeling the evenly spaced events in time requires a discrete-time type model. Therefore, discrete-time Hawkes processes are appealing for certain applications. However, there are few works on discrete-time Hawkes type models.


\cite{XZW2020} proposed for the first time a discrete-time self-exciting and mutually-exciting model analogous to Hawkes process.
More recently, the discrete-time self-exciting model was also applied to study the infection and death of COVID-19 in \cite{BDKVJ2021}.
\cite{WZ} extended the model of \cite{XZW2020} in the univariate case and studied its limit theorems. 
Following the model in \cite{WZ}, 
let $\alpha(t):\mathbb{N}\rightarrow\mathbb{R}_{+}$
be a positive function on $\mathbb{N}$ and define $X_{0}=N_{0}=0$.
We define $\Vert\alpha\Vert_{1}:=\sum_{t=1}^{\infty}\alpha(t)$
as the $\ell_{1}$ norm of $\alpha$. 
Conditional on $X_{t-1},X_{t-2},\ldots,X_{1}$, 
we define $Z_{t}$ as a Poisson random variable
with mean
\begin{equation}\label{eq:lambda_t}
\lambda_{t}=\nu+\sum_{s=1}^{t-1}\alpha(s)X_{t-s},
\end{equation}
and define
\begin{equation}\label{eq:compoundpoisson}
X_{t}=\sum_{j=1}^{Z_{t}}\ell_{t,j},
\end{equation}
where $\ell_{t,j}$ are positive random variables
that are i.i.d. in both $t$ and $j$. 
Finally, we define $N_{t}:=\sum_{s=1}^{t}Z_{s}$
and $L_{t}:=\sum_{s=1}^{t}X_{s}$.

Throughout the paper, we assume that $\Vert\alpha\Vert_{1}\mathbb{E}[\ell_{1,1}]<1$.
It can be derived that the law of large numbers hold:
\begin{equation}
\label{eq:lln}
\lim_{t\rightarrow\infty}\frac{N_{t}}{t}=\mu:=\frac{\nu}{1-\Vert\alpha\Vert_{1}\mathbb{E}[\ell_{1,1}]},
\qquad
\lim_{t\rightarrow\infty}\frac{L_{t}}{t}=\Tilde{\mu}:=\frac{\nu\mathbb{E}[\ell_{1,1}]}{1-\Vert\alpha\Vert_{1}\mathbb{E}[\ell_{1,1}]},
\end{equation}
in probability as $t\rightarrow\infty$, 
and the central limit theorem also holds, see \cite{WZ}:
\begin{align}
&\frac{1}{\sqrt{t}}
\left(N_{t}-\frac{\nu t}{1-\Vert\alpha\Vert_{1}\mathbb{E}[\ell_{1,1}]}\right)
\rightarrow \mathcal{N}\left(0,\frac{\nu(1+\Vert\alpha\Vert_{1}^{2}\text{Var}(\ell_{1,1}))}{(1-\Vert\alpha\Vert_{1}\mathbb{E}[\ell_{1,1}])^{3}}\right),
\\
&\frac{1}{\sqrt{t}}
\left(L_{t}-\frac{\nu\mathbb{E}[\ell_{1,1}] t}{1-\Vert\alpha\Vert_{1}\mathbb{E}[\ell_{1,1}]}\right)
\rightarrow \mathcal{N}\left(0,\frac{\nu\mathbb{E}[\ell_{1,1}^{2}]}{(1-\Vert\alpha\Vert_{1}\mathbb{E}[\ell_{1,1}])^{3}}\right),
\end{align}
in distribution as $t\rightarrow\infty$
under the assumptions that $$\lim_{t\rightarrow\infty}\frac{1}{\sqrt{t}}\sum_{u=1}^{t-1}\sum_{s=1+u}^{\infty}\alpha(s)=0,$$
and the first four moments of $\ell$ are finite.

In this paper, we are interested in studying the large and moderate deviations for the above discrete-time marked Hawkes process.
Before we proceed, we will briefly review the large deviation principle, the moderate deviation principle, and the existing results for Hawkes models.

\textbf{Other related literature.}
A discrete-time Hawkes-type model with 0-1 arrivals
was proposed by \cite{Seol} and the limit theorems were studied. 
Let $\left(X_n\right)^{\infty}_{n=1}$ be a sequence taking values on $\{0,1\}$ defined as follows.
Let $\hat{\mathbb{N}}=\mathbb{N}\bigcup\{0\}$ and assume that for $i\in \mathbb{N}$, $\alpha_i > 0$ is a given sequence of positive numbers and $\sum_{i=0}^{\infty}\alpha_i<1$.
(i) $X_1=1$ with probability $\alpha_0$ and $X_1 = 0$ otherwise.
(ii) Conditional on $X_1,X_2,...,X_{n-1}$, we have $X_n = 1$ with probability $\alpha_0 + \sum_{i=1}^{n-1}\alpha_{n-i}X_i$, and $X_n = 0$ otherwise.
Define $S_n := \sum_{i=1}^{n}X_i$. \cite{Seol} showed a law of large numbers theorem, i.e.
$$\frac{S_n}{n} \to \mu := \frac{\alpha_0}{1-\sum_{i=1}^{\infty}\alpha_i},$$
in probability as $n\to\infty$. 
In addition, with assumption $\sqrt{n}\sum_{i=n}^{\infty}\alpha_i\to0$ as $n\to\infty$ and $\frac{1}{\sqrt{n}}\sum_{i=1}^{n}i\alpha_i\to0$ as $n\to\infty$, the central limit theorem follows:
$$\frac{S_n-\mu n}{\sqrt{n}} \to \mathcal{N}\left(0,\frac{\mu(1-\mu)}{\left(1-\sum_{j=1}^{\infty}\alpha_j\right)^2}\right),$$ 
in distribution as $n\to\infty$.

\subsection{Large deviation principles}\label{Def:LDP}
Following \cite{Dembo}, we introduce the definition of large deviation principle. A family of probability measures $\{\mathbb{P}_{n}\}_{n\in\mathbb{N}}$ on a topological space $\left(X,\mathcal{T}\right)$ satisfies the large deviation principle with rate function $I(\cdot):X\to[0,\infty]$ and speed $a_n$ if $I$ is a lower semi-continuous function, $a_n:[0,\infty)\to[0,\infty)$ is a measurable function which increases to infinity, and the following inequalities hold for every Borel set $A$:
\begin{equation*}
-\inf_{x\in A^{o}}I(x)\leq\liminf_{n\rightarrow\infty}\frac{1}{a_{n}}\log \mathbb{P}_{n}(A)
\leq\limsup_{n\rightarrow\infty}\frac{1}{a_{n}}\log \mathbb{P}_{n}(A)\leq-\inf_{x\in\overline{A}}I(x).
\end{equation*}
Where $A^{o}$ is the interior of $A$ and $\overline{A}$ is the closure of A. We say that the rate function $I$ is good if for any $m\ge0$, the level set $\{x\in X:I(x)\le m \text{, }m\ge0\}$ is compact.
In addition to \cite{Dembo}, we also refer to \cite{VaradhanII} for a survey on large deviations.

\subsection{Large deviations for Hawkes processes}\label{sec:LDHP}
We first review some large deviations results for Hawkes processes
in the literature. We recall that the intensity of a unmarked linear Hawkes process 
with empty past history, i.e. $N(-\infty,0]=0$, is given by 
\begin{equation}
\label{lli:0}
\lambda_t := \nu +\int_{(0,t)}h(t-s)N(ds)
\end{equation}
where $\nu>0$. The integral in equation $\left(\ref{lli:0}\right)$ stands for $\int_{(0,t)}h(t-s)N(ds)=\sum_{\tau_i<t}h(t-\tau_i)$, where $(\tau_i)_{i\ge1}$ are the occurrences of the points before time $t$. If $\Vert h\Vert_{L^1}=\int_{0}^{\infty}h(t)dt<1$, the linear Hawkes process has an immigration-birth representation, and by ergodic theory, the law of large numbers for the linear Hawkes process (see, for instance, \cite{DV}) is derived as 
\begin{equation*}
    \lim_{t\to\infty}\dfrac{N_t}{t}=\dfrac{\nu}{1-\Vert h\Vert_{L^1}}.
\end{equation*}

\cite{Bordenave} showed that, if $0<\Vert h\Vert_{L^1}<1$ and $\int_{0}^{\infty}th(t)dt<\infty$, then $\mathbb{P}(\frac{N_t}{t}\in\cdot)$ satisfies the large deviation principle on $\mathbb{R}$ with the good rate function: 
\begin{equation}
	I(x) = \begin{cases}
	x\theta_{x}+\nu-\frac{\nu x}{\nu +\Vert h\Vert_{L^1} x}, & \text{if $x\in[0,\infty)$},\\
	\nu , & \text{if x=0},\\
	+\infty, & \text{otherwise},
	\end{cases}
\end{equation}
where $\theta \in (-\infty,\Vert h\Vert_{L^1}-1-\log\Vert h\Vert_{L^1}]$ and $\theta = \theta_{x}$ is the unique solution of $\mathbb{E}[e^{\theta S}]=\frac{x}{\nu+x\Vert h\Vert_{L^1}}$, $x>0$, where $S$ denotes the total number of descendants of an immigrant, including the immigrant itself. 

The large deviation principle of a marked linear Hawkes process with empty history can be found in \cite{Karabash}. Recall the notation of a general marked Hawkes process introduced in section~\ref{sec:introduction}, The intensity of a marked univariate linear Hawkes process is given by
\begin{equation}
\label{lli}
\lambda_t := \nu +\int_{(0,t)\times\mathbb{X}}h(t-s,\ell)N(ds,d\ell).
\end{equation}
Let $H(\ell):=\int_{0}^{\infty}h(t,
\ell)dt$ for any $\ell\in \mathbb{X}$. Assume that
\begin{equation}
\label{Ha}
\int_{\mathbb{X}}H(\ell)q(d\ell)<1.
\end{equation}
Under the above assumption, there exists a unique stationary version of the linear marked Hawkes process defined by equation~\eqref{lli}. And by ergodic theorem, a law of large numbers is derived as
$$\lim_{t\to\infty}\dfrac{N_t}{t}=\dfrac{\nu}{1-\mathbb{E}^{q}[H(\ell)]}.$$
If there exists some $\theta>0$, so that $\int_{\mathbb{X}}e^{\theta H(\ell)}q\left(d\ell\right)<\infty$. \cite{Karabash} proved that $\mathbb{P}\left(N_t/t\in\cdot\right)$ satisfies a large deviation principle with rate function:
\begin{equation}
	I(x) = \begin{cases}
	\theta_\star x-\nu\left(x_\star-1\right), & x \ge 0,\\
	+\infty, & x < 0,
	\end{cases}
\end{equation}
where $\theta_\star$ and $x_\star$ satisfy the following equations
\begin{equation}
	\begin{cases}
	x_\star = \mathbb{E}^q\left[e^{\theta_\star+\left(x_\star-1\right)H\left(\ell\right)}\right],\\
	\frac{x}{\nu}=x_\star+\frac{x}{\nu}\mathbb{E}^q\left[H\left(\ell\right)e^{\theta_\star+\left(x_\star-1\right)H\left(\ell\right)}\right].
	\end{cases}
\end{equation}

For nonlinear Hawkes processes, \cite{ZhuIII} established the level-3 large deviation principle first and then used the contraction principle to obtain the large deviation principle for $\mathbb{P}\left(N_t/t\in\cdot\right)$. \cite{ZhuIV} proved the large deviations for Markovian Hawkes processes and generalized the proof to the case when $h(\cdot)$ is a sum of exponentials starting with the case of exponential $h(\cdot)$. 

\subsection{Moderate deviation principles}
For any $\sqrt{n}\ll c_n\ll n$, a family of probability measures $\{\mathbb{P}_{n}\}_{n\in\mathcal{N}}$ on a topological space $\left(X,\mathcal{T}\right)$ satisfies a moderate deviation principle with rate function $J(\cdot):X\to[0,\infty]$ 
if $J$ is a lower semi-continuous function and for any Borel set $A$, 
\begin{equation*}
-\inf_{x\in A^{o}}J(x)\leq\liminf_{n\rightarrow\infty}\frac{n}{c^2_{n}}\log \mathbb{P}_{n}(A)
\leq\limsup_{n\rightarrow\infty}\frac{n}{c^2_{n}}\log \mathbb{P}_{n}(A)\leq-\inf_{x\in\overline{A}}J(x).
\end{equation*}
That is, $\mathbb{P}_{n}$ satisfies a large deviation principle with speed $\dfrac{c^2_n}{n}$.
For example, let $X_1,\cdots,X_n$ be a sequence of i.i.d random variables commonly distributed as $X$ and assume $\mathbb{E}\left[e^{\theta X}\right]<\infty$ for $\theta$ in some ball around the origin.
Then, $\mathbb{P}_{n}:=\mathbb{P}\left(\dfrac{1}{c_n}\sum^{n}_{i=1}X_i\in\cdot\right)$ satisfies a large deviation principle with speed $\dfrac{c^2_n}{n}$. Moderate deviations fills the gap between ordinary deviations approximated by the central limit theorem and large deviations.

\subsection{Moderate deviations for Hawkes processes}
\cite{ZhuII} proved the moderate deviation principle for a univariate linear Hawkes process, defined by formula~\eqref{lli:0} in section~\ref{sec:LDHP}. With the assumption $\sup_{t>0}t^{3/2}h(t)=C<\infty$, the moderate deviation principle holds with the rate function
\begin{equation}
    J(x)=\dfrac{x^2\left(1-\Vert h\Vert_{L^1}\right)^3}{2v}.
\end{equation}

The moderate deviation principles for a marked linear Hawkes process was studied in \cite{Seol3}. Recall the definition of a marked linear Hawkes process in section~\ref{sec:LDHP}. \cite{Seol3} showed the moderate deviation rate function is 
\begin{equation}
    J(x)=\dfrac{x^2\left(1-\mathbb{E}\left[H(\ell)\right]\right)^3}{2\nu\left(1+\text{Var}\left[H(\ell)\right]\right)},
\end{equation}
with assumptions, $\text{Var}\left[H(\ell)\right]<\infty$ and $\sup_{t>0}t^{3/2}\int_{\mathcal{X}}h(t,\ell)q(d\ell)\le C<\infty$.

\textbf{The other related literature.}
The large deviations of Cox-Ingersoll-Ross process with Hawkes jumps can be found in \cite{zhu2013limit}. \cite{ZBGG15} studied limit theorems of affine jump diffusion processes with Hawkes jumps. Gao and Zhu \cite{GZ2} studied large deviations of the Hawkes process with large initial intensity and also discussed the applications of the model to insurance and queue systems. 
\cite{Yao} studied the moderate deviation principle for multivariate unmarked linear Hawkes processes. And moderate deviation principles have been studied in mixing processes, Markov processes, martingales, etc. (see \cite{GAO1996,CHEN2001,Dembo1996})

\textbf{Organization of this paper.}
The rest of the paper is organized as follows. In section 2, we state our main results. The proof of the main results can be found in section 3.

\section{Main Results}
Recall the discrete-time Hawkes model introduced in section \ref{sec:introduction}, $N_{t}:=\sum_{s=1}^{t}Z_{s}$, $L_{t}:=\sum_{s=1}^{t}X_{s}$ where $Z_t$ is a Poisson random variable conditional on $\mathcal{F}_{t-1}$ with intensity $\lambda_t$ defined by equation \eqref{eq:lambda_t} and $X_t$ is a compound Poisson random variable defined by equation \eqref{eq:compoundpoisson}.
This section states the large deviations and moderate deviations of the discrete-time marked Hawkes process.

\subsection{Large deviations}
The formal definition of the large deviation principle has been introduced in section~\ref{Def:LDP}. For the discrete-time Hawkes process, we prove the following the large deviation principles.


\begin{theorem}\label{Thm:NLD}
$\mathbb{P}(N_{t}/t\in\cdot)$ satisfies a large deviation principle
with the rate function
\begin{equation}
I(x)=\sup_{\theta\leq\theta_{c}}\{\theta x-\Gamma(\theta)\},
\end{equation}
where
$\Gamma(\theta):=\nu(g(\theta)-1)$,
where $g(\theta)$ is the minimal solution to the equation
$x=\mathbb{E}[e^{\theta+\Vert\alpha\Vert_{1}\ell_{1,1}(x-1)}]$
for any $\theta\leq\theta_{c}$, where
$\theta_{c}=-\log\mathbb{E}[\Vert\alpha\Vert_{1}\ell_{1,1}e^{\Vert\alpha\Vert_{1}\ell_{1,1}(x_{c}-1)}]>0$,
where $x_{c}>1$ satisfies the equation
$x_{c}\mathbb{E}[\Vert\alpha\Vert_{1}\ell_{1,1}e^{\Vert\alpha\Vert_{1}\ell_{1,1}(x_{c}-1)}]
=\mathbb{E}[e^{\Vert\alpha\Vert_{1}\ell_{1,1}(x_{c}-1)}]$.
\end{theorem}


\begin{theorem}\label{Thm:LLD}
$\mathbb{P}(L_{t}/t\in\cdot)$ satisfies a large deviation principle
with the rate function
\begin{equation}
I_{L}(x)=\sup_{\theta\leq\theta_{c}}\{\theta x-\Gamma_{L}(\theta)\},
\end{equation}
where
$\Gamma_{L}(\theta):=\nu(g_{L}(\theta)-1)$,
where $g_{L}(\theta)$ is the minimal solution to the equation
$x=\mathbb{E}[e^{\theta\ell_{1,1}+\Vert\alpha\Vert_{1}\ell_{1,1}(x-1)}]$
for any $\theta\leq\theta_{c}$, where
$\theta_{c}>0$ satisfies the equation
$\mathbb{E}[\Vert\alpha\Vert_{1}\ell_{1,1}e^{\theta_{c}\ell_{1,1}+\Vert\alpha\Vert_{1}\ell_{1,1}(x_{c}-1)}]=1$,
where $x_{c}$ satisfies the equation $x_{c}=\mathbb{E}[e^{\theta_c\ell_{1,1}+\Vert\alpha\Vert_{1}\ell_{1,1}(x_{c}-1)}]$.
\end{theorem}

\subsection{Moderate deviations}

In terms of the moderate deviations of the discrete-time Hawkes process, we assume $\sup_{t>0}t^{3/2}\alpha(t) = C<\infty$. 
Recall the equation~\eqref{eq:lln}, where $\mu$ and $\Tilde{\mu}$ denote the limits in the law of large numbers for $N_t$ and $L_t$, respectively. 
We obtain the following moderate deviation principles for the discrete-time Hawkes process.

\begin{theorem}\label{Thm:NMDP}
For any Borel set $A$ and time sequence $c(t)$ such that $\sqrt{t}\ll c(t)\ll t$, we have the following moderate deviation principle.
\begin{align*}
-\inf_{x\in A^{o}}J(x)
&\leq\liminf_{t\rightarrow\infty}\frac{t}{c^2(t)}\log \mathbb{P}\left(\dfrac{N_t-\mu t}{c(t)}\in A\right)\\
&\leq\limsup_{t\rightarrow\infty}\frac{t}{c^2(t)}\log \mathbb{P}\left(\dfrac{N_t-\mu t}{c(t)}\in A\right)\leq-\inf_{x\in\overline{A}}J(x),
\end{align*}
where \begin{equation}
    J(x)
    =\dfrac{x^2\left(1-\mathbb{E}\left[\ell_{1,1}\Vert\alpha\Vert_1\right]\right)^3}{2\nu \left(1+\text{Var}(\ell_{1,1})\Vert\alpha\Vert^2_1\right)}.
\end{equation}
\end{theorem}

\begin{theorem}\label{Thm:LMDP}
For any Borel set $A$ and time sequence $c(t)$ such that $\sqrt{t}\ll c(t)\ll t$, we have the following moderate deviation principle.
\begin{align*}
-\inf_{x\in A^{o}}J(x)&\leq\liminf_{t\rightarrow\infty}\frac{t}{c^2_{t}}\log \mathbb{P}\left(\dfrac{L_t-\Tilde{\mu t}}{c(t)}\in A\right) \\
&\leq\limsup_{t\rightarrow\infty}\frac{t}{c^2(t)}\log \mathbb{P}\left(\dfrac{L_t-\Tilde{\mu} t}{c(t)}\in A\right)\leq-\inf_{x\in\overline{A}}J(x),
\end{align*}
where \begin{equation}
    J(x)
    =\dfrac{x^2\left(1-\mathbb{E}\left[\ell_{1,1}\Vert\alpha\Vert_1\right]\right)^3}{2\nu \mathbb{E}\left[\ell^2_{1,1}\right]}.
\end{equation}
\end{theorem}


\section{Proof of Main Results}
This section states the proof of our main results. Before we proceed, let's recall a version of G\"{a}rtner-Ellis theorem which will be used in our proof. 
\begin{theorem}[G\"{a}rtner-Ellis theorem (Theorem 2.3.6 \cite{Dembo})]\label{GE}
	Let $(\mu_n)_{n\in\mathbb{N}}$ be a sequence of probability measures on $\left(\mathbb{R},\mathbb{B}\left(\mathbb{R}\right)\right)$. Define the logarithmic moment generating function
	\begin{equation}
	    \Lambda_{n}\left(\theta\right):=\int_{\mathbb{R}}e^{\theta x_n}d\mu_n,
	\end{equation}
	and assume that for all $\theta\in\mathbb{R}$ a possibly infinite limit $\Lambda\left(\theta\right)$ in (\ref{LMGF}) 
	\begin{equation}\label{LMGF}
	\Lambda\left(\theta\right) := \lim\limits_{n\to\infty}\frac{1}{n}\log\left[\Lambda_{n}\left(n\theta\right)\right],
	\end{equation}
	exists and $0\in\mathcal{D}^\mathcal{O}_\Lambda$, where $\mathcal{D}^\mathcal{O}_\Lambda$ is the interior of $\mathcal{D}_\Lambda$ and $\mathcal{D}_\Lambda:=\left\{\theta\in\mathbb{R}:\Lambda\left(\theta\right)<\infty\right\}$.
	Suppose in addition that $\Lambda$ is lower semi-continuous on $\mathbb{R}$, differentiable on $\mathcal{D}^{\mathcal{O}}_\Lambda$, and $\Lambda$ is steep, i.e.
	$$\lim\limits_{n\to\infty}\left|\Lambda^{\prime}\left(\theta_n\right)\right|=\infty$$
	whenever $\theta_n\in\mathcal{D}^{\mathcal{O}}_\Lambda$, $\theta_n\to \theta\in\partial\mathcal{D}^{\mathcal{O}}_\Lambda$ as $n\to \infty$. Then $\left(\mu_n\right)_n\in\mathbb{N}$ satisfies the LDP with rate function $I$, which is Fenchel-Legendre transform of $\Lambda$,
	\begin{equation}
    I(x)=\sup_{\theta\in\mathbb{R}}\{\theta x-\Lambda(\theta)\}.
    \end{equation}

\end{theorem}

\subsection{Proof of large deviations}
\begin{proof}[Proof of Theorem \ref{Thm:NLD}]
	For any $\theta\in\mathbb{R}$, we can compute that
	\begin{align*}
	\mathbb{E}[e^{\theta N_{t}}]
	&=\mathbb{E}\left[e^{\theta N_{t-1}+\theta Z_{t}}\right]
	\\
	&=\mathbb{E}\left[\mathbb{E}\left[e^{\theta N_{t-1}+\theta Z_{t}}|\mathcal{F}_{t-1}\right]\right]
	\\
	&=\mathbb{E}\left[e^{\theta N_{t-1}}\mathbb{E}\left[e^{\theta Z_{t}}|\mathcal{F}_{t-1}\right]\right]
	\\
	&=\mathbb{E}\left[e^{\theta N_{t-1}+(e^{\theta}-1)\lambda_{t}}\right],
	\end{align*}
	where we used the fact that $Z_{t}$ is Poisson with parameter $\lambda_{t}$
	conditional on $\mathcal{F}_{t-1}$, the natural filtration up to time $t-1$.
	By the definition of $\lambda_{t}$, we have
	\begin{align*}
	\mathbb{E}\left[e^{\theta N_{t}}\right]
	&=\mathbb{E}\left[e^{\theta N_{t-1}+(e^{\theta}-1)\nu+(e^{\theta}-1)\sum_{s=1}^{t-1}\alpha(s)X_{t-s}}\right]
	\\
	&=e^{(e^{\theta}-1)\nu}\mathbb{E}\left[e^{\theta N_{t-1}+(e^{\theta}-1)\alpha(1)X_{t-1}+(e^{\theta}-1)\sum_{s=2}^{t-1}\alpha(s)X_{t-s}}\right]
	\\
	&=e^{(e^{\theta}-1)\nu}\mathbb{E}\left[e^{\theta N_{t-1}+\log\mathbb{E}[e^{(e^{\theta}-1)\alpha(1)\ell_{1,1}}]Z_{t-1}+(e^{\theta}-1)\sum_{s=2}^{t-1}\alpha(s)X_{t-s}}\right].
	\end{align*}
	Let $f_{0}(\theta)=\theta$
	and $f_{1}(\theta)=\theta+\log\mathbb{E}[e^{(e^{\theta}-1)\alpha(1)\ell_{1,1}}]$.
	Then, 
	\begin{align*}
	\mathbb{E}\left[e^{\theta N_{t}}\right]
	&=e^{(e^{\theta}-1)\nu}\mathbb{E}\left[e^{\theta N_{t-2}+f_{1}(\theta)Z_{t-1}+(e^{\theta}-1)\sum_{s=2}^{t-1}\alpha(s)X_{t-s}}\right]
	\\
	&=e^{(e^{\theta}-1)\nu}\mathbb{E}\left[e^{\theta N_{t-2}
		+(e^{f_{1}(\theta)}-1)\lambda_{t-1}+(e^{\theta}-1)\sum_{s=2}^{t-1}\alpha(s)X_{t-s}}\right].
	\end{align*}
	By the definition of $\lambda_{t-1}$, we get
	\begin{align*}
	&\mathbb{E}\left[e^{\theta N_{t}}\right]
	\\
	&=e^{(e^{\theta}-1)\nu}\mathbb{E}\left[e^{\theta N_{t-2}
		+(e^{f_{1}(\theta)}-1)(\nu+\sum_{s=1}^{t-2}\alpha(s)X_{t-1-s})+(e^{\theta}-1)\sum_{s=2}^{t-1}\alpha(s)X_{t-s}}\right]
	\\
	&=e^{(e^{\theta}-1)\nu+(e^{f_{1}(\theta)}-1)\nu}\mathbb{E}\left[e^{\theta N_{t-2}
		+(e^{f_{1}(\theta)}-1)\sum_{s=1}^{t-2}\alpha(s)X_{t-1-s}+(e^{\theta}-1)\sum_{s=2}^{t-1}\alpha(s)X_{t-s}}\right]
	\\
	&=e^{(e^{\theta}-1)\nu+(e^{f_{1}(\theta)}-1)\nu}\mathbb{E}\Bigg[e^{
		\theta N_{t-2}
		+(e^{f_{1}(\theta)}-1)\alpha(1)X_{t-2}+(e^{\theta}-1)\alpha(2)X_{t-2}}
	\\
	&\qquad\qquad\qquad
	\cdot e^{(e^{f_{1}(\theta)}-1)\sum_{s=2}^{t-2}\alpha(s)X_{t-1-s}+(e^{\theta}-1)\sum_{s=3}^{t-1}\alpha(s)X_{t-s}}\Bigg]
	\\
	&=e^{(e^{\theta}-1)\nu+(e^{f_{1}(\theta)}-1)\nu}\mathbb{E}\Bigg[e^{
		\theta N_{t-2}
		+\log\mathbb{E}[e^{((e^{f_{1}(\theta)}-1)\alpha(1)+(e^{\theta}-1)\alpha(2))\ell_{1,1}}] Z_{t-2}}
	\\
	&\qquad\qquad\qquad
	\cdot e^{(e^{f_{1}(\theta)}-1)\sum_{s=2}^{t-2}\alpha(s)X_{t-1-s}+(e^{\theta}-1)\sum_{s=3}^{t-1}\alpha(s)X_{t-s}}\Bigg].
	\end{align*}
	By induction on $t$, we get
	\begin{align*}
	\mathbb{E}\left[e^{\theta N_{t}}\right]
	&=e^{\nu((e^{f_{0}(\theta)}-1)+(e^{f_{1}(\theta)}-1)
		+\cdots+(e^{f_{t-2}(\theta)}-1))}
	\mathbb{E}[e^{f_{t-1}(\theta)N_{1}}]
	\\
	&=e^{\nu((e^{f_{0}(\theta)}-1)+(e^{f_{1}(\theta)}-1)
		+\cdots+(e^{f_{t-2}(\theta)}-1))}e^{\nu(e^{f_{t-1}(\theta)}-1)},
	\end{align*}
	where $f_{0}(\theta)=\theta$,
	$f_{1}(\theta)=\theta+\log\mathbb{E}[e^{(e^{\theta}-1)\alpha(1)\ell_{1,1}}]$,
	and
	\begin{equation}
	f_{2}(\theta)=\theta+\log\mathbb{E}\left[e^{((e^{f_{1}(\theta)}-1)\alpha(1)+(e^{\theta}-1)\alpha(2))\ell_{1,1}}\right],
	\end{equation}
	and more generally, for every $s\geq 1$,
	\begin{equation}
	f_{s}(\theta)=\theta+\log\mathbb{E}\left[e^{((e^{f_{s-1}(\theta)}-1)\alpha(1)+(e^{f_{s-2}(\theta)}-1)\alpha(2)
		+\cdots+(e^{f_{0}(\theta)}-1)\alpha(s))\ell_{1,1}}\right].
	\end{equation}
	This implies that
	\begin{equation}
	\lim_{t\rightarrow\infty}\frac{1}{t}\log\mathbb{E}\left[e^{\theta N_{t}}\right]
	=\nu(e^{f_{\infty}(\theta)}-1),
	\end{equation}
	where
	\begin{equation}\label{eq:finf}
	f_{\infty}(\theta)=\theta+\log\mathbb{E}\left[e^{(e^{f_{\infty}(\theta)}-1)\Vert\alpha\Vert_{1}\ell_{1,1}}\right].
	\end{equation}
	Let $x=e^{f_{\infty}(\theta)}$. Thus, equation \eqref{eq:finf} can be rewritten as 
	\begin{equation}\label{eq:x}
		x = \mathbb{E}\left[e^{\theta+(x-1)\Vert\alpha\Vert_{1}\ell_{1,1}}\right].
	\end{equation}
	It means we need to show the solution of equation~\eqref{eq:x} exists.
	
	First, it is not hard to see when $\theta \le 0$, $e^{f_t(\theta)}$ is decreasing in $t$ and $0< e^{f_t(\theta)}\le 1$. Thus, the limit of $e^{f_t(\theta)}$ converges to a finite limit $x_{\star}$ as $t\to\infty$, which satisfies equation~\eqref{eq:x}. 
	
	Next, when $\theta>0$, $f_t(\theta)$ is increasing in $t$. We need to determine for what values of $\theta$ the solution of equation~\eqref{eq:x} exists. Let
	\begin{equation}
		G(x) = \mathbb{E}\left[e^{\theta+(x-1)\Vert\alpha\Vert_{1}\ell_{1,1}}\right]-x.
	\end{equation}
	It is easy to see that $G(x)$ is increasing in $\theta$ and $G^{\prime\prime}(x)>0$. If $\theta=0$, then $G(x)=\mathbb{E}\left[e^{(x-1)\Vert\alpha\Vert_{1}\ell_{1,1}}\right]-x$ satisfies $G(1)=0$. 
	Moreover, $G^{\prime}(1)=\mathbb{E}\left[\Vert\alpha\Vert_{1}\ell_{1,1}\right]-1$. By the assumption $\mathbb{E}\left[\Vert\alpha\Vert_{1}\ell_{1,1}\right]<1$, we have $G^{\prime}(1)<0$. It implies $\min_{x>1}G(x)<0$. 
	Hence, there exists some critical $\theta_c>0$ such that $\min_{x>1}G(x)=0$. In other words, with $\theta_c$, we can find critical value $x_c$ such that $G(x_c)=G^{\prime}(x_c)=0$. Thus, we can find
	\begin{equation}\label{eq:thetaC}
		\theta_c=-\log\mathbb{E}\left[\Vert\alpha\Vert_{1}\ell_{1,1}e^{(x_c-1)\Vert\alpha\Vert_{1}\ell_{1,1}}\right]
	\end{equation}
	where $x_c>1$ satisfies $x\mathbb{E}\left[\Vert\alpha\Vert_{1}\ell_{1,1}e^{(x-1)\Vert\alpha\Vert_{1}\ell_{1,1}}\right]=\mathbb{E}\left[e^{(x-1)\Vert\alpha\Vert_{1}\ell_{1,1}}\right]$. Therefore, equation~\eqref{eq:x} has finite solutions if and only if $\theta<\theta_c$. 
	
	
	$G(x)$ is strictly convex in $x$. Hence, there are at most two solutions for equation~\eqref{eq:x}. 
	When $0<\theta < \theta_{c}$, equation~\eqref{eq:x} has two solutions. It's not hard to check $G(1)=e^\theta-1>0$ and $G'(1)=\mathbb{E}\left[\Vert\alpha\Vert_{1}\ell_{1,1}e^\theta\right]-1<0$.
	$f_t(\theta)$ is increasing in $t$ and for $t=0$, $e^{f_0(\theta)}=e^\theta>1$ and
	\begin{equation}
		G(e^\theta)=e^\theta\left[\mathbb{E}\left[e^{(e^\theta-1)\Vert\alpha\Vert_{1}\ell_{1,1}}\right]-1\right]>0.
	\end{equation}
	Thus, as $t\to\infty$, $e^{f_t(\theta)}$ converges to a finite limit. It must converges to $x_\star$ which is the smaller solution of equation~\eqref{eq:x}.
	
	Similarly, when $\theta<0$, we can check $G(1)<0$ and $G'(1)<0$. $f_t(\theta)$ is decreasing in $t$ and at $t=0$, $e^{f_0(\theta)}=e^{\theta}<1$ with $G(e^\theta)<0$. Thus, as $t\to\infty$, $e^{f_t(\theta)}$ converges to a finite limit. It must converges to $x_\star$ which is also the smaller solution of equation~\eqref{eq:x}. 

	If $\theta>\theta_{c}$, then $\lim_{t\rightarrow\infty}\frac{1}{t}\log\mathbb{E}\left[e^{\theta N_{t}}\right]=\infty$.
	
	Finally, we need to check the essential smoothness condition of $\nu\left(e^{f_\infty(\theta)}-1\right)$.
	\begin{equation}\label{eq:esmooth}
		f^{\prime}_\infty(\theta)=\dfrac{\mathbb{E}\left[e^{(e^{f_{\infty}(\theta)}-1)\Vert\alpha\Vert_{1}\ell_{1,1}}\right]}{\mathbb{E}\left[e^{(e^{f_{\infty}(\theta)}-1)\Vert\alpha\Vert_{1}\ell_{1,1}}\right]-e^{f_\infty(\theta)}\mathbb{E}\left[\Vert\alpha\Vert_{1}\ell_{1,1}e^{(e^{f_{\infty}(\theta)}-1)\Vert\alpha\Vert_{1}\ell_{1,1}}\right]}
	\end{equation}
	By equation~\eqref{eq:thetaC}, it is not hard to find $|f^{\prime}_\infty(\theta)| \to \infty$ as $\theta \to \theta_{c}$, the conclusion then follows G\"{a}rtner-Ellis theorem.
\end{proof}
\begin{proof}[Proof of Theorem \ref{Thm:LLD}]
	For any $\theta\in\mathbb{R}$, we can compute that
	\begin{align*}
	\mathbb{E}[e^{\theta L_{t}}]
	&=\mathbb{E}\left[e^{\theta L_{t-1}+\theta X_{t}}\right]
	\\
	&=\mathbb{E}\left[\mathbb{E}\left[e^{\theta L_{t-1}+\theta X_{t}}|\mathcal{F}_{t-1}\right]\right]
	\\
	&=\mathbb{E}\left[e^{\theta L_{t-1}}\mathbb{E}\left[e^{\theta X_{t}}|\mathcal{F}_{t-1}\right]\right]
	\\
	&=\mathbb{E}\left[e^{\theta L_{t-1}+(\mathbb{E}[e^{\theta\ell_{1,1}}]-1)\lambda_{t}}\right],
	\end{align*}
	where we used the fact that $X_{t}$ is compound Poisson with intensity $\lambda_{t}$
	conditional on $\mathcal{F}_{t-1}$, the natural filtration up to time $t-1$.
	By the definition of $\lambda_{t}$, we have
	\begin{align*}
	\mathbb{E}\left[e^{\theta L_{t}}\right]
	&=\mathbb{E}\left[e^{\theta L_{t-1}+(\mathbb{E}[e^{\theta\ell_{1,1}}]-1)\nu+(\mathbb{E}[e^{\theta\ell_{1,1}}]-1)\sum_{s=1}^{t-1}\alpha(s)X_{t-s}}\right]
	\\
	&=e^{(\mathbb{E}[e^{\theta\ell_{1,1}}]-1)\nu}\mathbb{E}\left[e^{\theta L_{t-1}+(\mathbb{E}[e^{\theta\ell_{1,1}}]-1)\alpha(1)X_{t-1}+(\mathbb{E}[e^{\theta\ell_{1,1}}]-1)\sum_{s=2}^{t-1}\alpha(s)X_{t-s}}\right]
	\\
	&=e^{(\mathbb{E}[e^{\theta\ell_{1,1}}]-1)\nu}\mathbb{E}\left[e^{\theta L_{t-2}+
		(\theta+(\mathbb{E}[e^{\theta\ell_{1,1}}]-1)\alpha(1))X_{t-1}+(\mathbb{E}[e^{\theta\ell_{1,1}}]-1)\sum_{s=2}^{t-1}\alpha(s)X_{t-s}}\right]
	\\
	&=e^{(\mathbb{E}[e^{\theta\ell_{1,1}}]-1)\nu}\mathbb{E}\left[e^{\theta L_{t-2}+
		(\mathbb{E}[e^{(\theta+(\mathbb{E}[e^{\theta\ell_{1,1}}]-1)\alpha(1))\ell_{1,1}}]-1)\lambda_{t-1}+(\mathbb{E}[e^{\theta\ell_{1,1}}]-1)\sum_{s=2}^{t-1}\alpha(s)X_{t-s}}\right].
	\end{align*}
	By the definition of $\lambda_{t-1}$, we get
	\begin{align*}
	&\mathbb{E}\left[e^{\theta L_{t}}\right]
	\\
	&=e^{(\mathbb{E}[e^{\theta\ell_{1,1}}]-1)\nu}
	e^{(\mathbb{E}[e^{(\theta+(\mathbb{E}[e^{\theta\ell_{1,1}}]-1)\alpha(1))\ell_{1,1}}]-1)\nu}
	\\
	&\qquad\cdot
	\mathbb{E}\left[e^{\theta L_{t-2}+
		(\mathbb{E}[e^{(\theta+(\mathbb{E}[e^{\theta\ell_{1,1}}]-1)\alpha(1))\ell_{1,1}}]-1)\sum_{s=1}^{t-2}\alpha(s)X_{t-1-s}+(\mathbb{E}[e^{\theta\ell_{1,1}}]-1)\sum_{s=2}^{t-1}\alpha(s)X_{t-s}}\right].
	\end{align*}
	
	By induction on $t$, we get
	\begin{equation}
	\mathbb{E}\left[e^{\theta L_{t}}\right]
	=e^{\nu((g_{0}(\theta)-1)+(g_{1}(\theta)-1)
		+\cdots+(g_{t-1}(\theta)-1))}=\mathbb{E}\left[e^{\nu(\sum^{t-1}_{s=0}g_s(\theta)-1)}\right],
	\end{equation}
	where $g_{0}(\theta)=\mathbb{E}[e^{\theta\ell_{1,1}}]$,
	$g_{1}(\theta)=\mathbb{E}[e^{(\theta+(g_{0}(\theta)-1)\alpha(1))\ell_{1,1}}]$,
	and more generally, for every $s\geq 1$,
	\begin{align}
	g_{s}(\theta)
	&=\mathbb{E}\left[e^{(\theta+(g_{s-1}(\theta)-1)\alpha(1)+(g_{s-2}(\theta)-1)\alpha(2)
		+\cdots+(g_{0}(\theta)-1)\alpha(s))\ell_{1,1}}\right] \nonumber\\
	&=\mathbb{E}\left[e^{(\theta+\sum^{s}_{i=1}\alpha(i)(g_{s-i}(\theta)-1))\ell_{1,1}}\right].
	\end{align}
	This implies that
	\begin{equation}
	\lim_{t\rightarrow\infty}\frac{1}{t}\log\mathbb{E}\left[e^{\theta L_{t}}\right]
	=\nu(g_{\infty}(\theta)-1),
	\end{equation}
	where
	\begin{equation}
	g_{\infty}(\theta)=\mathbb{E}\left[e^{(\theta+(g_{\infty}(\theta)-1)\Vert\alpha\Vert_{1})\ell_{1,1}}\right].
	\end{equation}
	Similar as before, we have $\lim_{t\rightarrow\infty}\frac{1}{t}\log\mathbb{E}\left[e^{\theta L_{t}}\right]=\nu(g_{L}(\theta)-1)$,
	where $g_{L}(\theta)$ is the minimal solution to the equation
	$x=\mathbb{E}[e^{\theta\ell_{1,1}+\Vert\alpha\Vert_{1}\ell_{1,1}(x-1)}]$
	for any $\theta\leq\theta_{c}$, where
	$\theta_{c}$ satisfies the equation
	$\mathbb{E}[\Vert\alpha\Vert_{1}\ell_{1,1}e^{\theta_{c}\ell_{1,1}+\Vert\alpha\Vert_{1}\ell_{1,1}(x_{c}-1)}]=1$,
	where $x_{c}$ satisfies the equation $x_{c}=\mathbb{E}[e^{\theta_c\ell_{1,1}+\Vert\alpha\Vert_{1}\ell_{1,1}(x_{c}-1)}]$.
	If $\theta>\theta_{c}$, then $\lim_{t\rightarrow\infty}\frac{1}{t}\log\mathbb{E}\left[e^{\theta L_{t}}\right]=\infty$.
	We can check the essential smoothness condition
	similar as before.
	\begin{equation}
	    g^{\prime}_{\infty}(\theta)=\dfrac{\mathbb{E}\left[e^{(\theta+(g_{\infty}(\theta)-1)\Vert\alpha\Vert_{1})\ell_{1,1}}\Vert\alpha\Vert_1\ell_{1,1}\right]}{1-\mathbb{E}[\Vert\alpha\Vert_{1}\ell_{1,1}e^{\theta\ell_{1,1}+\Vert\alpha\Vert_{1}\ell_{1,1}(x_{c}-1)}]},
	\end{equation}
	it is not hard to find $|g^{\prime}_\infty(\theta)| \to \infty$ as $\theta \to \theta_{c}$, the conclusion then follows G\"{a}rtner-Ellis theorem.
\end{proof}

\subsection{Proof of moderate deviations}
\begin{proof}[Proof of Theorem~\ref{Thm:NMDP}]
First, for any $\theta\in\mathbb{R}$, we prove that 
$$\lim_{t\to\infty}\dfrac{t}{c^2(t)}\log{\mathbb{E}\left[e^{\frac{c(t)}{t}\theta (N_t-\mu t)}\right]}
=\dfrac{\nu\theta^2\left(1+\text{Var}(\ell_{1,1})\Vert\alpha\Vert^2_1\right)}{2\left(1-\mathbb{E}\left[\ell_{1,1}\Vert\alpha\Vert_1\right]\right)^3},$$
where $\mu$ is defined by equation~\eqref{eq:lln}.

By the proof of Theorem~\ref{Thm:NLD}, we get
\begin{align*}
    \mathbb{E}\left[e^{\frac{c(t)}{t}\theta N_t}\right] 
    &= e^{\nu\left((e^{f_0(\theta_t)}-1)+(e^{f_1(\theta_t)}-1)+\cdots+(e^{f_{t-2}(\theta_t)}-1)\right)}\mathbb{E}\left[e^{f_{t-1}(\theta_t)N_1}\right]\\
    &=e^{\nu\left((e^{f_0(\theta_t)}-1)+(e^{f_1(\theta_t)}-1)+\cdots+(e^{f_{t-1}(\theta_t)}-1)\right)},
\end{align*}
where $f_0(\theta_t)=\theta_t:=\frac{c(t)}{t}\theta$, $f_1(\theta_t)=\theta_t+\log{\mathbb{E}\left[e^{(e^{\theta_t}-1)\alpha(1)\ell_{1,1}} \right]}$, and $f_2(\theta_t)=\theta_t
    +\log{\mathbb{E}\left[e^{((e^{f_1(\theta_t)}-1)\alpha(1)
    +(e^{f_0(\theta_t)}-1)\alpha(2))\ell_{1,1}}\right]}$.
More generally, for every $s\ge 1$,
\begin{align*}
    f_s(\theta_t)
    &=\theta_t
    +\log{\mathbb{E}\left[e^{((e^{f_{s-1}(\theta_t)}-1)\alpha(1)
    +(e^{f_{s-2}(\theta_t)}-1)\alpha(2)+\cdots+(e^{f_{0}(\theta_t)}-1)\alpha(s))\ell_{1,1}}\right]}.
\end{align*}
Then we can rewrite the above equation such that
\begin{align*}
    e^{f_s(\theta_t)} 
    &= e^{\theta_t}
    \mathbb{E}\left[e^{((e^{f_{s-1}(\theta_t)}-1)\alpha(1)
    +(e^{f_{s-2}(\theta_t)}-1)\alpha(2)+\cdots+(e^{f_{0}(\theta_t)}-1)\alpha(s))\ell_{1,1}}\right].
\end{align*}
Let's define $G_t(s)=e^{f_s(\theta_t)}-1$ so that $G_t(s)=
    \mathbb{E}\left[e^{\theta_t+\ell_{1,1}\sum^{s}_{i=1}\alpha(i)G_t(s-i)}\right]-1$. 
Then we have
\begin{align}\label{eq:char Nt}
    \mathbb{E}\left[e^{\frac{c(t)}{t}\theta N_t}\right] 
    &=e^{\nu\sum^{t-1}_{s=0}G_t(s)}.
\end{align}

We write $G_t(s)$ instead of $G(s)$ to indicate its dependence on t because of the term, $\dfrac{c(t)}{t}$. 
As the proof of Theorem~\ref{Thm:NLD} shows and $\dfrac{c(t)}{t}$ is sufficient small so that we have $\dfrac{c(t)}{t}\theta\le \theta_c$ where $\theta_c=-\log{\mathbb{E}\left[\Vert\alpha\Vert_1\ell_{1,1}e^{\Vert\alpha\Vert_1\ell_{1,1}(x_c-1)}\right]}$ and as $s\to \infty$, we get that
$G_t(\infty)$ is the minimal solution to the equation $x_t= \mathbb{E}\left[e^{\frac{c(t)}{t}\theta+\ell_{1,1}\sum^{\infty}_{i=1}\alpha(i)x_t}\right]-1$. Because $\mathbb{E}\left[\ell_{1,1}\right]\Vert\alpha\Vert_{1}<1$, it is easy to see that $x_t=O((c(t)/t))$. Because $x_t=O((c(t)/t))$, we have $G_t(s)=O((c(t)/t))$ uniformly in s. By Taylor's expansion,
\begin{align}\label{eq:taylor Gt}
    G_t(s)&=\dfrac{c(t)\theta}{t} + \mathbb{E}[\ell_{1,1}\sum^{s}_{i=1}\alpha(i)G_{t}(s-i)] \nonumber\\
    &\quad+\dfrac{1}{2}\left(\dfrac{c(t)\theta}{t}\right)^2+\dfrac{1}{2}\mathbb{E}\left[\left(\ell_{1,1}\sum^{s}_{i=1}\alpha(i)G_{t}(s-i)\right)^2\right] \nonumber\\
    &\quad+\dfrac{c(t)\theta}{t}\mathbb{E}[\ell_{1,1}\sum^{s}_{i=1}\alpha(i)G_{t}(s-i)]+O\left(\left(\dfrac{c(t)}{t}\right)^3\right).
\end{align}
Now, let
\begin{equation}\label{eq:GtG1G2}
    G_t(s)=\dfrac{c(t)\theta}{t}G_1(s)+\left(\dfrac{c(t)\theta}{t}\right)^2 G_2(s)+\epsilon_t(s),
\end{equation}
where $G_1(s)$ satisfies
\begin{align}\label{eq:G1}
    G_1(s) = 1 + \mathbb{E}[\ell_{1,1}\sum^{s}_{i=1}\alpha(i)G_{1}(s-i)],
\end{align}
and $G_1(0)=1$, and $G_2(s)$ satisfies
\begin{align}\label{eq:G2}
    G_2(s) = \mathbb{E}[\ell_{1,1}\sum^{s}_{i=1}\alpha(i)G_{2}(s-i)] + \dfrac{1}{2}+(G_1(s)-1)+\dfrac{1}{2}\mathbb{E}\left[\left(\ell_{1,1}\sum^{s}_{i=1}\alpha(i)G_1(s-i)\right)^2\right],
\end{align}
and $G_2(0)=1/2$. Then we can substitute $G_t(s)$ in terms of equation~\eqref{eq:G1} and equation~\eqref{eq:G2} into the left and right side of equation~\eqref{eq:taylor Gt} so that we get $\epsilon_t(s)=O\left(\left(\dfrac{c(t)}{t}\right)^3\right)$.

By equation~\eqref{eq:char Nt},
\begin{align*}
    \dfrac{t}{c^2(t)}\log{\mathbb{E}\left[e^{\frac{c(t)}{t}\theta (N_t-\mu t)}\right]} =\dfrac{t}{c^2(t)}\left( \nu\sum^{t-1}_{s=0}G_t(s)\right)-\dfrac{\mu\theta t}{c(t)}.
\end{align*}
Then we can rewrite the above equation in terms of equation~\eqref{eq:GtG1G2},
\begin{align*}
    \dfrac{t}{c^2(t)}\log{\mathbb{E}\left[e^{\frac{c(t)}{t}\theta (N_t-\mu t)}\right]} 
    &=\dfrac{t}{c^2(t)}\left( \nu\sum^{t-1}_{s=0}\left(\dfrac{c(t)\theta}{t}G_1(s)+\left(\dfrac{c(t)\theta}{t}\right)^2 G_2(s)+\epsilon_t(s)\right)\right)-\dfrac{\mu\theta t}{c(t)}\\
    &=\dfrac{\nu\theta}{c(t)}\sum^{t-1}_{s=0}G_1(s)-\dfrac{\mu\theta t}{c(t)} + \dfrac{\nu\theta^2}{t}\sum^{t-1}_{s=0}G_2(s) + O\left(\left(\dfrac{c(t)}{t^2}\right)\right).
\end{align*}

Now let's compute $\sum^{t-1}_{s=0}G_1(s)$. By equation~\eqref{eq:G1},
\begin{align*}
    \sum^{t-1}_{s=0}G_1(s)&=\sum^{t-1}_{s=1}G_1(s) + 1
    = 1+\sum^{t-1}_{s=1} 1 +\sum^{t-1}_{s=1}\mathbb{E}\left[\ell_{1,1}\sum^{s}_{i=1}\alpha(i)G_{1}(s-i)\right]\\
    &= t + \sum^{t-1}_{s=1}\mathbb{E}\left[\ell_{1,1}\sum^{s}_{i=1}\alpha(i)G_{1}(s-i)\right]\\
    &= t + \mathbb{E}\left[\ell_{1,1}\sum^{t-1}_{i=1}\sum^{t-1}_{s=i}\alpha(i)G_{1}(s-i)\right] \\
    &= t + \mathbb{E}\left[\ell_{1,1}\sum^{t-1}_{i=1}\sum^{t-1-i}_{j=0}\alpha(i)G_{1}(j)\right]\\
    &= t + \mathbb{E}\left[\ell_{1,1}\sum^{t-1}_{i=1}\alpha(i)\left(\sum^{t-1}_{j=0}G_{1}(j)-\sum^{t-1}_{j=t-i}G_{1}(j)\right)\right].
\end{align*}
After rewriting the above equation, we get
\begin{align*}
    \sum^{t-1}_{s=0}G_1(s)
    -\mathbb{E}\left[\ell_{1,1}\sum^{t-1}_{i=1}\alpha(i)\sum^{t-1}_{j=0}G_{1}(j)\right]
    &= t -\mathbb{E}\left[\ell_{1,1}\sum^{t-1}_{i=1}\alpha(i)\sum^{t-1}_{j=t-i}G_{1}(j)\right]\\
    \sum^{t-1}_{s=0}G_1(s) &= \dfrac{t -\mathbb{E}\left[\ell_{1,1}\sum^{t-1}_{i=1}\alpha(i)\sum^{t-1}_{j=t-i}G_{1}(j)\right]}{1-\mathbb{E}\left[\ell_{1,1}\sum^{t-1}_{i=1}\alpha(i)\right]}.
\end{align*}
Hence, 
\begin{align}\label{eq:limterm1}
    \dfrac{\nu\theta}{c(t)}\sum^{t-1}_{s=0}G_1(s)-\dfrac{\mu\theta t}{c(t)}
    &=\dfrac{\nu\theta}{c(t)}\left(\dfrac{t -\mathbb{E}\left[\ell_{1,1}\sum^{t-1}_{i=1}\alpha(i)\sum^{t-1}_{j=t-i}G_{1}(j)\right]}{1-\mathbb{E}\left[\ell_{1,1}\sum^{t-1}_{i=1}\alpha(i)\right]}-\dfrac{t}{1-\Vert\alpha\Vert_{1}\mathbb{E}[\ell_{1,1}]}\right) \nonumber\\
    &= \dfrac{\nu\theta}{c(t)}\left(\dfrac{t}{1-\mathbb{E}\left[\ell_{1,1}\sum^{t-1}_{i=1}\alpha(i)\right]}-\dfrac{t}{1-\Vert\alpha\Vert_{1}\mathbb{E}[\ell_{1,1}]}\right) \nonumber\\
    &\quad-\dfrac{\nu\theta}{c(t)}\dfrac{\mathbb{E}\left[\ell_{1,1}\sum^{t-1}_{i=1}\alpha(i)\sum^{t-1}_{j=t-i}G_{1}(j)\right]}{1-\mathbb{E}\left[\ell_{1,1}\sum^{t-1}_{i=1}\alpha(i)\right]}.
\end{align}
For the first term in equation~\eqref{eq:limterm1}, we can compute
\begin{align*}
    &\left|\dfrac{\nu\theta}{c(t)}\left(
    \dfrac{t}{1-\mathbb{E}\left[\ell_{1,1}\sum^{t-1}_{i=1}\alpha(i)\right]}-\dfrac{t}{1-\Vert\alpha\Vert_{1}\mathbb{E}[\ell_{1,1}]}\right)
    \right|\\
    &= \left|\dfrac{\nu\theta t}{c(t)}
    \dfrac{-\Vert\alpha\Vert_{1}\mathbb{E}[\ell_{1,1}]+\mathbb{E}\left[\ell_{1,1}\sum^{t-1}_{i=1}\alpha(i)\right]}{\left(1-\mathbb{E}\left[\ell_{1,1}\sum^{t-1}_{i=1}\alpha(i)\right]\right)\left(1-\Vert\alpha\Vert_{1}\mathbb{E}[\ell_{1,1}]\right)}
    \right|\\
    &\le \left|\dfrac{\nu\theta t}{c(t)}
    \dfrac{\mathbb{E}\left[\ell_{1,1}\sum^{\infty}_{i=t}\alpha(i)\right]}{\left(1-\Vert\alpha\Vert_{1}\mathbb{E}[\ell_{1,1}]\right)^2}
    \right|=\dfrac{\nu t}{c(t)}\left|
    \dfrac{\theta\mathbb{E}\left[\ell_{1,1}\sum^{\infty}_{i=t}\alpha(i)\right]}{\left(1-\Vert\alpha\Vert_{1}\mathbb{E}[\ell_{1,1}]\right)^2}
    \right|.
\end{align*}
According to the assumption $\sup_{t>0}t^{3/2}\alpha(t) = C<\infty$, we can find $\sum^{\infty}_{i=t}\alpha(i)\le \sum^{\infty}_{i=t}C i^{-3/2}< 2C (t-1)^{-1/2}$. Thus,
\begin{align*}
    \dfrac{\nu t}{c(t)}\left|
    \dfrac{\theta\mathbb{E}\left[\ell_{1,1}\sum^{\infty}_{i=t}\alpha(i)\right]}{\left(1-\Vert\alpha\Vert_{1}\mathbb{E}[\ell_{1,1}]\right)^2}
    \right|
    \le 
    \dfrac{\nu C}{c(t)}\left|
    \dfrac{\theta\mathbb{E}\left[\ell_{1,1}\right](\dfrac{2t}{\sqrt{t-1}})}{\left(1-\Vert\alpha\Vert_{1}\mathbb{E}[\ell_{1,1}]\right)^2}
    \right|\to 0, \text{\ as\ } t\to \infty.
\end{align*}
By Lemma~\ref{lemma:limG1}, $G_1(t)$ is uniformly bounded. Then for the second term in equation~\eqref{eq:limterm1},
\begin{align*}
    \limsup_{t\to\infty}\left|\dfrac{\nu\theta}{c(t)}\dfrac{\mathbb{E}\left[\ell_{1,1}\sum^{t-1}_{i=1}\alpha(i)\sum^{t-1}_{j=t-i}G_{1}(j)\right]}{1-\mathbb{E}\left[\ell_{1,1}\sum^{t-1}_{i=1}\alpha(i)\right]}\right|
    &\le G_1(\infty)\limsup_{t\to\infty}\dfrac{\nu\left|\theta\right|}{c(t)}\dfrac{\mathbb{E}\left[\ell_{1,1}\sum^{t-1}_{i=1}(i-1)\alpha(i)\right]}{1-\Vert\alpha\Vert_{1}\mathbb{E}[\ell_{1,1}]}\\
    &= G_1(\infty)\limsup_{t\to\infty}\dfrac{\nu\left|\theta\right|}{c(t)}\dfrac{\sum^{t-1}_{i=1}\mathbb{E}\left[\ell_{1,1}(i-1)\alpha(i)\right]}{1-\Vert\alpha\Vert_{1}\mathbb{E}[\ell_{1,1}]}.
\end{align*}
And we can compute
\begin{align*}
    \limsup_{t\to\infty}\dfrac{\nu\left|\theta\right|}{c(t)}\dfrac{\sum^{t-1}_{i=1}\mathbb{E}\left[\ell_{1,1}(i-1)\alpha(i)\right]}{1-\Vert\alpha\Vert_{1}\mathbb{E}[\ell_{1,1}]}
    &=\limsup_{t\to\infty}\dfrac{\nu\left|\theta\right|}{c(t)}\dfrac{\sum^{t-1}_{i=1}\mathbb{E}\left[\ell_{1,1}(i-1)\alpha(i)\right]}{1-\Vert\alpha\Vert_{1}\mathbb{E}[\ell_{1,1}]}\\
    &\le \limsup_{t\to\infty}\dfrac{\nu\left|\theta\right|}{c(t)}\dfrac{\sum^{t-1}_{i=1}\mathbb{E}\left[\ell_{1,1}\dfrac{C}{\sqrt{i}}\right]}{1-\Vert\alpha\Vert_{1}\mathbb{E}[\ell_{1,1}]}\\
    &\le \limsup_{t\to\infty}\dfrac{\dfrac{2C\sqrt{t}\nu\left|\theta\right|}{c(t)}\mathbb{E}\left[\ell_{1,1}\right]}{1-\Vert\alpha\Vert_{1}\mathbb{E}[\ell_{1,1}]}=0
\end{align*}
Hence,
\begin{align*}
   \lim_{t\to\infty}\left[ \dfrac{\nu\theta}{c(t)}\sum^{t-1}_{s=0}G_1(s)-\dfrac{\mu\theta t}{c(t)}\right]
    &=0.
\end{align*}
Furthermore, we also get 
\begin{align}
\label{eq:limG1}
    \lim_{t\to\infty}\sum^{t-1}_{s=0}\dfrac{1}{t}G_1(s)=\dfrac{1}{1-\mathbb{E}\left[\ell_{1,1}\right]\Vert\alpha\Vert_1}.
\end{align}

According to Lemma~\ref{lemma:limG2}, $G_2(t)$ is uniformly bounded in $t$. Then we can compute
\begin{align}
\label{eq:limG2}
    \lim_{t\to\infty} \dfrac{1}{t}\sum^{t-1}_{s=0}G_2(s)
    &=\dfrac{1}{2}\dfrac{1+2\left(\dfrac{1}{1-\mathbb{E}\left[\ell_{1,1}\right]\Vert\alpha\Vert_1}-1\right)+\left(\dfrac{\mathbb{E}\left[\left(\ell_{1,1}\sum^{t}_{i=1}\alpha(i)\right)^2\right]}{(1-\mathbb{E}\left[\ell_{1,1}\right]\Vert\alpha\Vert_1)^2}\right)}{1-\mathbb{E}\left[\ell_{1,1}\right]\Vert\alpha\Vert_1}\nonumber\\
    &=\dfrac{1+\text{Var}(\ell_{1,1})\Vert\alpha\Vert^2_1}{2\left(1-\mathbb{E}\left[\ell_{1,1}\right]\Vert\alpha\Vert_1\right)^3}.
\end{align}
Now we can have
\begin{align}
    \lim_{t\to\infty} \dfrac{\nu\theta^2}{t}\sum^{t-1}_{s=0}G_2(s)
    &=\dfrac{\nu\theta^2\left(1+\text{Var}(\ell_{1,1})\Vert\alpha\Vert^2_1\right)}{2\left(1-\mathbb{E}\left[\ell_{1,1}\right]\Vert\alpha\Vert_1\right)^3}.
\end{align}

Thus, we can prove
\begin{align*}
    \lim_{t\to\infty} \dfrac{t}{c^2(t)}\log{\mathbb{E}\left[e^{\frac{c(t)}{t}\theta (N_t-\mu t)}\right]} 
    &=\dfrac{\nu\theta^2\left(1+\text{Var}(\ell_{1,1})\Vert\alpha\Vert^2_1\right)}{2\left(1-\mathbb{E}\left[\ell_{1,1}\right]\Vert\alpha\Vert_1\right)^3}.
\end{align*}

Applying the G\"{a}rtner-Ellis theorem, we conclude that, for any Borel set $A$,
\begin{align}
    -\inf_{x\in A^{o}} J(x)
    &\le \liminf_{t\to\infty}\dfrac{t}{c^2(t)}\log{\mathbb{P}\left(\dfrac{N_t-\mu t}{c(t)}\in A\right)} \nonumber \\
    &\le \limsup_{t\to\infty}\dfrac{t}{c^2(t)}\log{\mathbb{P}\left(\dfrac{N_t-\mu t}{c(t)}\in A\right)}\le -\inf_{x\in\Bar{A}}J(x),
\end{align}
where 
\begin{equation}
    J(x)=\sup_{\theta\in\mathbb{R}}\left\{\theta x-\dfrac{\nu\theta^2\left(1+\text{Var}(\ell_{1,1})\Vert\alpha\Vert^2_1\right)}{2\left(1-\mathbb{E}\left[\ell_{1,1}\Vert\alpha\Vert_1\right]\right)^3}\right\}
    =\dfrac{x^2\left(1-\mathbb{E}\left[\ell_{1,1}\Vert\alpha\Vert_1\right]\right)^3}{2\nu \left(1+\text{Var}(\ell_{1,1})\Vert\alpha\Vert^2_1\right)}.
\end{equation}
\end{proof}

\begin{proof}[Proof of Theorem~\ref{Thm:LMDP}]
First, let's prove  
$$\lim_{t\to \infty} \dfrac{t}{c^2(t)}\log{\mathbb{E}\left[e^{\frac{c(t)}{t}\theta (L_t-\Tilde{\mu} t)}\right]}=\dfrac{\nu\theta^2\mathbb{E}\left[\ell^2_{1,1}\right]}{2\left(1-\mathbb{E}\left[\ell_{1,1}\right]\Vert\alpha\Vert_1\right)^3},$$
where $\Tilde{\mu}$ is defined in equation~\eqref{eq:lln}.

As the proof of Theorem~\ref{Thm:LLD} shows, we can find
\begin{align}\label{eq:char Lt}
    \mathbb{E}[e^{\frac{c(t)\theta}{t}L_t}]=e^{\nu\left(\left(g_0(\theta_t)-1\right)+\cdots+\left(g_{t-1}(\theta_t)-1\right)\right)}=e^{\nu\sum^{t-1}_{s=0}(g_s(\theta_t)-1)},
\end{align}
where $\theta_t=\frac{c(t)}{t}\theta$, $g_0(\theta_t)=\mathbb{E}\left[e^{\theta_t\ell_{1,1}}\right]$, 
$g_1(\theta_t)=\mathbb{E}\left[e^{\left(\theta_t+\left(g_0(\theta_t)-1 \right)\alpha(1)\right)\ell_{1,1}}\right]$, and in general for every $s\geq 1$,
\begin{align}
	g_{s}(\theta_t)
	&=\mathbb{E}\left[e^{(\theta_t+(g_{s-1}(\theta_t)-1)\alpha(1)+(g_{s-2}(\theta_t)-1)\alpha(2)
		+\cdots+(g_{0}(\theta_t)-1)\alpha(s))\ell_{1,1}}\right] \nonumber\\
	&=\mathbb{E}\left[e^{(\theta_t+\sum^{s}_{i=1}(g_{i-1}(\theta_t)-1)\alpha(i))\ell_{1,1}}\right].
\end{align}
Thus, $\Tilde{G}_t(0)=\mathbb{E}\left[e^{\theta_t\ell_{1,1}}\right]-1$, 
$\Tilde{G}_t(s)=
    \mathbb{E}\left[e^{\left(\theta_t+\sum^{s}_{i=1}\alpha(i)\Tilde{G}_t(s-i)\right)\ell_{1,1}}\right]
    -1$ for $s\ge 1$.

Because of $\dfrac{c(t)}{t}$, we write $\Tilde{G}_t(s)$ instead of $\Tilde{G}(s)$ to indicate its dependence on t. 
According to the proof of Theorem~\ref{Thm:LLD}, $\dfrac{c(t)}{t}$ is sufficient small so that we have $\dfrac{c(t)}{t}\theta\le \theta_c$, where $\theta_c$ satisfies $\mathbb{E}\left[\Vert\alpha\Vert_1\ell_{1,1}e^{\theta_c\ell_{1,1}+\Vert\alpha\Vert_1\ell_{1,1}(x_c-1)}\right]=1$ 
and as $s\to \infty$, we get that 
$\Tilde{G}_t(\infty)$ is the minimal solution to the equation $\Tilde{x}_t= \mathbb{E}\left[e^{\frac{c(t)}{t}\theta\ell_{1,1}+\ell_{1,1}\sum^{\infty}_{i=1}\alpha(i)\Tilde{x}_t}\right]-1$. Because $\mathbb{E}\left[\ell_{1,1}\right]\Vert\alpha\Vert_{1}<1$, it is easy to see that $\Tilde{x}_t=O((c(t)/t))$. Because $\Tilde{x}_t=O((c(t)/t))$, we have $\Tilde{G}_t(s)=O((c(t)/t))$ uniformly in s. By Taylor's expansion,
\begin{align}\label{eq:taylor LGt}
    \Tilde{G}_t(s)&=\dfrac{c(t)\theta}{t}\mathbb{E}\left[\ell_{1,1}\right] + \mathbb{E}[\ell_{1,1}\sum^{s}_{i=1}\alpha(i)\Tilde{G}_{t}(s-i)] \nonumber\\
    &\quad+\dfrac{1}{2}\left(\dfrac{c(t)\theta}{t}\right)^2\mathbb{E}\left[\ell^2_{1,1}\right]+\dfrac{1}{2}\mathbb{E}\left[\left(\ell_{1,1}\sum^{s}_{i=1}\alpha(i)\Tilde{G}_{t}(s-i)\right)^2\right] \nonumber\\
    &\quad+\dfrac{c(t)\theta}{t}\mathbb{E}[\ell^2_{1,1}\sum^{s}_{i=1}\alpha(i)\Tilde{G}_{t}(s-i)]+O\left(\left(\dfrac{c(t)}{t}\right)^3\right).
\end{align}
We can let
\begin{equation}\label{eq:LGtG1G2}
    \Tilde{G}_t(s)=\dfrac{c(t)\theta}{t}\Tilde{G}_1(s)+\left(\dfrac{c(t)\theta}{t}\right)^2 \Tilde{G}_2(s)+\epsilon_t(s),
\end{equation}
where $G_1(s)$ satisfies
\begin{align}\label{eq:LG1}
    \Tilde{G}_1(s) = \mathbb{E}\left[\ell_{1,1}\right] + \mathbb{E}[\ell_{1,1}\sum^{s}_{i=1}\alpha(i)\Tilde{G}_{1}(s-i)],
\end{align}
$\Tilde{G}_1(0)=\mathbb{E}\left[\ell_{1,1}\right]$ 
and $\Tilde{G}_2(s)$ satisfies
\begin{align}\label{eq:LG2}
    \Tilde{G}_2(s) 
    &= \mathbb{E}[\ell_{1,1}\sum^{s}_{i=1}\alpha(i)\Tilde{G}_{2}(s-i)] + \dfrac{\mathbb{E}\left[\ell^2_{1,1}\right]}{2} \\
    &\quad+\mathbb{E}[\ell^2_{1,1}\sum^{s}_{i=1}\alpha(i)\Tilde{G}_{1}(s-i)]
    +\dfrac{1}{2}\mathbb{E}\left[\left(\ell_{1,1}\sum^{s}_{i=1}\alpha(i)\Tilde{G}_{1}(s-i)\right)^2\right],
\end{align}
$\Tilde{G}_2(0)=\dfrac{\mathbb{E}\left[\ell^2_{1,1}\right]}{2}$. Then we can substitute $\Tilde{G}_t(s)$ in terms of equation~\eqref{eq:LG1} and equation~\eqref{eq:LG2} into the left and right side of equation~\eqref{eq:taylor LGt} and we find $\epsilon_t(s)=O\left(\left(\dfrac{c(t)}{t}\right)^3\right)$.

By equation~\eqref{eq:char Lt},
\begin{align*}
    \dfrac{t}{c^2(t)}\log{\mathbb{E}\left[e^{\frac{c(t)}{t}\theta (L_t-\Tilde{\mu} t)}\right]} =\dfrac{t}{c^2(t)}\left( \nu\sum^{t-1}_{s=0}\Tilde{G}_t(s)\right)-\dfrac{\Tilde{\mu}\theta t}{c(t)},
\end{align*}
Then by equation~\eqref{eq:LGtG1G2}, we can rewrite the above equation,
\begin{align*}
    \dfrac{t}{c^2(t)}\log{\mathbb{E}\left[e^{\frac{c(t)}{t}\theta (L_t-\Tilde{\mu} t)}\right]} 
    &=\dfrac{t}{c^2(t)}\left( \nu\sum^{t-1}_{s=0}\left(\dfrac{c(t)\theta}{t}\Tilde{G}_1(s)+\left(\dfrac{c(t)\theta}{t}\right)^2 \Tilde{G}_2(s)+\epsilon_t(s)\right)\right)-\dfrac{\Tilde{\mu}\theta t}{c(t)}\\
    &=\dfrac{\nu\theta}{c(t)}\sum^{t-1}_{s=0}\Tilde{G}_1(s)-\dfrac{\Tilde{\mu}\theta t}{c(t)} + \dfrac{\nu\theta^2}{t}\sum^{t-1}_{s=0}\Tilde{G}_2(s) + O\left(\left(\dfrac{c(t)}{t^2}\right)\right).
\end{align*}

Now let's compute $\sum^{t-1}_{s=0}\Tilde{G}_1(s)$. By equation~\eqref{eq:LG1},
\begin{align*}
    \sum^{t-1}_{s=0}\Tilde{G}_1(s)
    &=\sum^{t-1}_{s=1}\Tilde{G}_1(s) + \mathbb{E}\left[\ell_{1,1}\right]
    = \mathbb{E}\left[\ell_{1,1}\right]+\sum^{t-1}_{s=1} \mathbb{E}\left[\ell_{1,1}\right] +\sum^{t-1}_{s=1}\mathbb{E}\left[\ell_{1,1}\sum^{s}_{i=1}\alpha(i)\Tilde{G}_{1}(s-i)\right]\\
    &= t\mathbb{E}\left[\ell_{1,1}\right] + \sum^{t-1}_{s=1}\mathbb{E}\left[\ell_{1,1}\sum^{s}_{i=1}\alpha(i)\Tilde{G}_{1}(s-i)\right]\\
    &= t\mathbb{E}\left[\ell_{1,1}\right] + \mathbb{E}\left[\ell_{1,1}\sum^{t-1}_{i=1}\sum^{t-1}_{s=i}\alpha(i)\Tilde{G}_{1}(s-i)\right] \\
    &= t\mathbb{E}\left[\ell_{1,1}\right] + \mathbb{E}\left[\ell_{1,1}\sum^{t-1}_{i=1}\sum^{t-1-i}_{j=0}\alpha(i)\Tilde{G}_{1}(j)\right]\\
    &= t\mathbb{E}\left[\ell_{1,1}\right] + \mathbb{E}\left[\ell_{1,1}\sum^{t-1}_{i=1}\alpha(i)\left(\sum^{t-1}_{j=0}\Tilde{G}_{1}(j)-\sum^{t-1}_{j=t-i}\Tilde{G}_{1}(j)\right)\right].
\end{align*}
After rewriting the above equation, we get
\begin{align*}
    \sum^{t-1}_{s=0}\Tilde{G}_1(s)
    -\mathbb{E}\left[\ell_{1,1}\sum^{t-1}_{i=1}\alpha(i)\sum^{t-1}_{j=0}\Tilde{G}_{1}(j)\right]
    &= t\mathbb{E}\left[\ell_{1,1}\right] -\mathbb{E}\left[\ell_{1,1}\sum^{t-1}_{i=1}\alpha(i)\sum^{t-1}_{j=t-i}\Tilde{G}_{1}(j)\right]\\
    \sum^{t-1}_{s=0}\Tilde{G}_1(s) &= \dfrac{t\mathbb{E}\left[\ell_{1,1}\right] -\mathbb{E}\left[\ell_{1,1}\sum^{t-1}_{i=1}\alpha(i)\sum^{t-1}_{j=t-i}\Tilde{G}_{1}(j)\right]}{1-\mathbb{E}\left[\ell_{1,1}\sum^{t-1}_{i=1}\alpha(i)\right]}.
\end{align*}
And we can compute
\begin{align}\label{eq:limLterm1}
    \dfrac{\nu\theta}{c(t)}\sum^{t-1}_{s=0}\Tilde{G}_1(s)-\dfrac{\Tilde{\mu}\theta t}{c(t)}
    &=\dfrac{\nu\theta}{c(t)}\left(\dfrac{t\mathbb{E}\left[\ell_{1,1}\right] -\mathbb{E}\left[\ell_{1,1}\sum^{t-1}_{i=1}\alpha(i)\sum^{t-1}_{j=t-i}\Tilde{G}_{1}(j)\right]}{1-\mathbb{E}\left[\ell_{1,1}\sum^{t-1}_{i=1}\alpha(i)\right]}-\dfrac{t\mathbb{E}\left[\ell_{1,1}\right]}{1-\Vert\alpha\Vert_{1}\mathbb{E}[\ell_{1,1}]}\right) \nonumber\\
    &= \dfrac{\nu\theta}{c(t)}\left(\dfrac{t\mathbb{E}\left[\ell_{1,1}\right]}{1-\mathbb{E}\left[\ell_{1,1}\sum^{t-1}_{i=1}\alpha(i)\right]}-\dfrac{t\mathbb{E}\left[\ell_{1,1}\right]}{1-\Vert\alpha\Vert_{1}\mathbb{E}[\ell_{1,1}]}\right) \nonumber\\
    &\quad-\dfrac{\nu\theta}{c(t)}\dfrac{\mathbb{E}\left[\ell_{1,1}\sum^{t-1}_{i=1}\alpha(i)\sum^{t-1}_{j=t-i}\Tilde{G}_{1}(j)\right]}{1-\mathbb{E}\left[\ell_{1,1}\sum^{t-1}_{i=1}\alpha(i)\right]}.
\end{align}
For the first term on the right hand side of the equation~\eqref{eq:limLterm1}, we can compute
\begin{align*}
    &\left|\dfrac{\nu\theta}{c(t)}\left(
    \dfrac{t\mathbb{E}\left[\ell_{1,1}\right]}{1-\mathbb{E}\left[\ell_{1,1}\sum^{t-1}_{i=1}\alpha(i)\right]}-\dfrac{t\mathbb{E}\left[\ell_{1,1}\right]}{1-\Vert\alpha\Vert_{1}\mathbb{E}[\ell_{1,1}]}\right)
    \right|\\
    &= \left|\dfrac{\nu\theta t\mathbb{E}\left[\ell_{1,1}\right]}{c(t)}
    \dfrac{-\Vert\alpha\Vert_{1}\mathbb{E}[\ell_{1,1}]+\mathbb{E}\left[\ell_{1,1}\sum^{t-1}_{i=1}\alpha(i)\right]}{\left(1-\mathbb{E}\left[\ell_{1,1}\sum^{t-1}_{i=1}\alpha(i)\right]\right)\left(1-\Vert\alpha\Vert_{1}\mathbb{E}[\ell_{1,1}]\right)}
    \right|\\
    &\le \left|\dfrac{\nu\theta t\mathbb{E}\left[\ell_{1,1}\right]}{c(t)}
    \dfrac{\mathbb{E}\left[\ell_{1,1}\sum^{\infty}_{i=t}\alpha(i)\right]}{\left(1-\Vert\alpha\Vert_{1}\mathbb{E}[\ell_{1,1}]\right)^2}
    \right|=\dfrac{\nu t}{c(t)}\left|
    \dfrac{\theta\mathbb{E}\left[\ell_{1,1}\sum^{\infty}_{i=t}\alpha(i)\right]}{\left(1-\Vert\alpha\Vert_{1}\mathbb{E}[\ell_{1,1}]\right)^2}
    \right|.
\end{align*}
According to the assumption $\sup_{t>0}t^{3/2}\alpha(t) = C<\infty$, we get $\sum^{\infty}_{i=t}\alpha(i)\le \sum^{\infty}_{i=t}C i^{-3/2}< 2C (t-1)^{-1/2}$. Therefore,
\begin{align*}
    \dfrac{\nu t\mathbb{E}\left[\ell_{1,1}\right]}{c(t)}\left|
    \dfrac{\theta\mathbb{E}\left[\ell_{1,1}\sum^{\infty}_{i=t}\alpha(i)\right]}{\left(1-\Vert\alpha\Vert_{1}\mathbb{E}[\ell_{1,1}]\right)^2}
    \right|
    \le 
    \dfrac{\nu C}{c(t)}\left|
    \dfrac{\theta\left(\mathbb{E}\left[\ell_{1,1}\right]\right)^2(\dfrac{2t}{\sqrt{t-1}})}{\left(1-\Vert\alpha\Vert_{1}\mathbb{E}[\ell_{1,1}]\right)^2}
    \right|\to 0, \text{\ as\ } t\to \infty.
\end{align*}
Next, by Lemma~\ref{lemma:limG1}, $\Tilde{G}_2(t)$ is uniformly bounded in $t$, then for the second term on the right hand side of the equation~\eqref{eq:limLterm1}, we have
\begin{align*}
    \limsup_{t\to\infty}\left|\dfrac{\nu\theta}{c(t)}\dfrac{\mathbb{E}\left[\ell_{1,1}\sum^{t-1}_{i=1}\alpha(i)\sum^{t-1}_{j=t-i}\Tilde{G}_{1}(j)\right]}{1-\mathbb{E}\left[\ell_{1,1}\sum^{t-1}_{i=1}\alpha(i)\right]}\right|
    &\le \Tilde{G}_1(\infty)\limsup_{t\to\infty}\dfrac{\nu\left|\theta\right|}{c(t)}\dfrac{\mathbb{E}\left[\ell_{1,1}\sum^{t-1}_{i=1}(i-1)\alpha(i)\right]}{1-\Vert\alpha\Vert_{1}\mathbb{E}[\ell_{1,1}]}\\
    &= \Tilde{G}_1(\infty)\limsup_{t\to\infty}\dfrac{\nu\left|\theta\right|}{c(t)}\dfrac{\sum^{t-1}_{i=1}\mathbb{E}\left[\ell_{1,1}(i-1)\alpha(i)\right]}{1-\Vert\alpha\Vert_{1}\mathbb{E}[\ell_{1,1}]}.
\end{align*}
Then we can get
\begin{align*}
    \limsup_{t\to\infty}\dfrac{\nu\left|\theta\right|}{c(t)}\dfrac{\sum^{t-1}_{i=1}\mathbb{E}\left[\ell_{1,1}(i-1)\alpha(i)\right]}{1-\Vert\alpha\Vert_{1}\mathbb{E}[\ell_{1,1}]}
    &=\limsup_{t\to\infty}\dfrac{\nu\left|\theta\right|}{c(t)}\dfrac{\sum^{t-1}_{i=1}\mathbb{E}\left[\ell_{1,1}(i-1)\alpha(i)\right]}{1-\Vert\alpha\Vert_{1}\mathbb{E}[\ell_{1,1}]}\\
    &\le \limsup_{t\to\infty}\dfrac{\nu\left|\theta\right|}{c(t)}\dfrac{\sum^{t-1}_{i=1}\mathbb{E}\left[\ell_{1,1}\dfrac{C}{\sqrt{i}}\right]}{1-\Vert\alpha\Vert_{1}\mathbb{E}[\ell_{1,1}]}\\
    &\le \limsup_{t\to\infty}\dfrac{\dfrac{2C\sqrt{t}\nu\left|\theta\right|}{c(t)}\mathbb{E}\left[\ell_{1,1}\right]}{1-\Vert\alpha\Vert_{1}\mathbb{E}[\ell_{1,1}]}=0.
\end{align*}
Therefore, we can compute
\begin{align*}
   \lim_{t\to\infty}\left[ \dfrac{\nu\theta}{c(t)}\sum^{t-1}_{s=0}\Tilde{G}_1(s)-\dfrac{\Tilde{\mu}\theta t}{c(t)}\right]
    &=0.
\end{align*}
Furthermore, we can show
\begin{align}
\label{eq:limLG1}
    \lim_{t\to\infty}\dfrac{1}{t}\sum^{t-1}_{s=0}\Tilde{G}_1(t)=\dfrac{\mathbb{E}\left[\ell_{1,1}\right]}{1-\mathbb{E}\left[\ell_{1,1}\Vert\alpha\Vert_1\right]}.
\end{align}

For $\dfrac{1}{t}\sum^{t-1}_{s=0}\Tilde{G}_2(s)$, by Lemma~\ref{lemma:limG2}, $\Tilde{G}_2(t)$ is uniformly bounded in $t$. Thus, we can compute 
\begin{align}
\label{eq:limLG2}
    \lim_{t\to\infty} \dfrac{1}{t}\sum^{t-1}_{s=0}\Tilde{G}_2(s)
    &=\dfrac{1}{2}\dfrac{\mathbb{E}\left[\ell^2_{1,1}\right]+2\left(\dfrac{\mathbb{E}\left[\ell_{1,1}\right]\Vert\alpha\Vert_1\mathbb{E}\left[\ell^2_{1,1}\right]}{1-\mathbb{E}\left[\ell_{1,1}\Vert\alpha\Vert_1\right]}\right)+\mathbb{E}\left[\ell^2_{1,1}\right]\left(\dfrac{\mathbb{E}\left[\ell_{1,1}\right]\Vert\alpha\Vert_1}{1-\mathbb{E}\left[\ell_{1,1}\Vert\alpha\Vert_1\right]}\right)^2}{1-\mathbb{E}\left[\ell_{1,1}\Vert\alpha\Vert_1\right]}\nonumber\\
    &=\dfrac{\mathbb{E}\left[\ell^2_{1,1}\right]}{2\left(1-\mathbb{E}\left[\ell_{1,1}\Vert\alpha\Vert_1\right]\right)^3}.
\end{align}
Noe we have
\begin{align}
    \lim_{t\to\infty} \dfrac{\nu\theta^2}{t}\sum^{t-1}_{s=0}\Tilde{G}_2(s)
    &=\dfrac{\nu\theta^2\mathbb{E}\left[\ell^2_{1,1}\right]}{2\left(1-\mathbb{E}\left[\ell_{1,1}\Vert\alpha\Vert_1\right]\right)^3}.
\end{align}

Thus, we can derive
\begin{align*}
    \lim_{t\to\infty} \dfrac{t}{c^2(t)}\log{\mathbb{E}\left[e^{\frac{c(t)}{t}\theta (L_t-\Tilde{\mu} t)}\right]} 
    &=\dfrac{\nu\theta^2\mathbb{E}\left[\ell^2_{1,1}\right]}{2\left(1-\mathbb{E}\left[\ell_{1,1}\Vert\alpha\Vert_1\right]\right)^3}.
\end{align*}

Applying the G\"{a}rtner-Ellis theorem see \cite{Dembo}, we conclude that, for any Borel set $A$,
\begin{align}
    -\inf_{x\in A^{o}} J(x)
    &\le \liminf_{t\to\infty}\dfrac{t}{c^2(t)}\log{\mathbb{P}\left(\dfrac{L_t-\Tilde{\mu} t}{c(t)}\in A\right)} \nonumber \\
    &\le \limsup_{t\to\infty}\dfrac{t}{c^2(t)}\log{\mathbb{P}\left(\dfrac{L_t-\Tilde{\mu} t}{c(t)}\in A\right)}\le -\inf_{x\in\Bar{A}}J(x),
\end{align}
where 
\begin{equation}
    J(x)=\sup_{\theta\in\mathbb{R}}\left\{\theta x-\dfrac{\nu\theta^2\mathbb{E}\left[\ell^2_{1,1}\right]}{2\left(1-\mathbb{E}\left[\ell_{1,1}\Vert\alpha\Vert_1\right]\right)^3}\right\}
    =\dfrac{x^2\left(1-\mathbb{E}\left[\ell_{1,1}\Vert\alpha\Vert_1\right]\right)^3}{2\nu \mathbb{E}\left[\ell^2_{1,1}\right]}.
\end{equation}
\end{proof}

\begin{lemma}\label{lemma:limG1}
For any $s\in\mathbb{N}$, 
\begin{align*}
    G_1(s)\le\dfrac{1}{1-\mathbb{E}\left[\ell_{1,1}\Vert\alpha\Vert_1\right]},
\end{align*}
where $G_1(s) = 1 + \mathbb{E}[\ell_{1,1}\sum^{s}_{i=1}\alpha(i)G_{1}(s-i)]$ for $s\ge 1$ and $G_1(0)=1$.
and 
\begin{align*}
    \Tilde{G}_1(s)\le\dfrac{\mathbb{E}\left[\ell_{1,1}\right]}{1-\mathbb{E}\left[\ell_{1,1}\Vert\alpha\Vert_1\right]},
\end{align*}
where $\Tilde{G}_1(0)=\mathbb{E}\left[\ell_{1,1}\right]$ and for $s\ge 1$, $\Tilde{G}_1(s) = \mathbb{E}\left[\ell_{1,1}\right] + \mathbb{E}[\ell_{1,1}\sum^{s}_{i=1}\alpha(i)\Tilde{G}_{1}(s-i)]$.
\end{lemma}
\begin{proof}[Proof of Lemma~\ref{lemma:limG1}]
We prove Lemma~\ref{lemma:limG1} by induction on $s$.
By assumption $\mathbb{E}\left[\ell_{1,1}\Vert\alpha\Vert_1\right]< 1$,  $G_1(0)\le\dfrac{1}{1-\mathbb{E}\left[\ell_{1,1}\Vert\alpha\Vert_1\right]}$. Now, let's assume that $G_1(s)\le\dfrac{1}{1-\mathbb{E}\left[\ell_{1,1}\Vert\alpha\Vert_1\right]}$. Then we can compute,
\begin{align*}
    G_1(s+1)
    &=1+\mathbb{E}[\ell_{1,1}\sum^{s+1}_{i=1}\alpha(i)G_{1}(s+1-i)]\\
    &\le 1+\mathbb{E}[\ell_{1,1}\sum^{s}_{i=1}\alpha(i)\dfrac{1}{1-\mathbb{E}\left[\ell_{1,1}\Vert\alpha\Vert_1\right]}]\\
    &=\dfrac{1-\mathbb{E}\left[\ell_{1,1}\Vert\alpha\Vert_1\right]+\mathbb{E}[\ell_{1,1}\sum^{s}_{i=1}\alpha(i)]}{1-\mathbb{E}\left[\ell_{1,1}\Vert\alpha\Vert_1\right]}\\
    &\le\dfrac{1}{1-\mathbb{E}\left[\ell_{1,1}\Vert\alpha\Vert_1\right]}
\end{align*}
Hence,  we  proved  that  for  every $s\in\mathbb{N}$, 
$ G_1(s)\le\dfrac{1}{1-\mathbb{E}\left[\ell_{1,1}\Vert\alpha\Vert_1\right]}$.
Similarly, we can show $\Tilde{G}_1(s)\le\dfrac{\mathbb{E}\left[\ell_{1,1}\right]}{1-\mathbb{E}\left[\ell_{1,1}\Vert\alpha\Vert_1\right]}$
\end{proof}

\begin{lemma}\label{lemma:limG2}
For any $s\in\mathbb{N}$, 
\begin{align*}
    G_2(s)\le\dfrac{1+\text{Var}(\ell_{1,1})\Vert\alpha\Vert^2_1}{2\left(1-\mathbb{E}\left[\ell_{1,1}\Vert\alpha\Vert_1\right]\right)^3},
\end{align*}
where $G_2(s) = \mathbb{E}[\ell_{1,1}\sum^{s}_{i=1}\alpha(i)G_{2}(s-i)] + \dfrac{1}{2}+(G_1(s)-1)+\dfrac{1}{2}\mathbb{E}\left[\left(\ell_{1,1}\sum^{s}_{i=1}\alpha(i)G_1(s-i)\right)^2\right]$ for $s\ge 1$ and $G_2(0)=1/2$.
And 
\begin{align*}
    \Tilde{G}_2(s)\le \dfrac{\mathbb{E}\left[\ell^2_{1,1}\right]}{2\left(1-\mathbb{E}\left[\ell_{1,1}\Vert\alpha\Vert_1\right]\right)^3}
\end{align*}
where $\Tilde{G}_2(0)=\dfrac{\mathbb{E}\left[\ell^2_{1,1}\right]}{2}$ 
and for $s\ge 1$, 
\begin{align*}
\Tilde{G}_2(s) 
    &= \mathbb{E}[\ell_{1,1}\sum^{s}_{i=1}\alpha(i)\Tilde{G}_{2}(s-i)] + \dfrac{\mathbb{E}\left[\ell^2_{1,1}\right]}{2} \\
    &\quad+\mathbb{E}[\ell^2_{1,1}\sum^{s}_{i=1}\alpha(i)\Tilde{G}_{1}(s-i)]
    +\dfrac{1}{2}\mathbb{E}\left[\left(\ell_{1,1}\sum^{s}_{i=1}\alpha(i)\Tilde{G}_{1}(s-i)\right)^2\right].
\end{align*}

\end{lemma}

\begin{proof}[Proof of Lemma~\ref{lemma:limG2}]
We prove Lemma~\ref{lemma:limG2} by induction on s.
By assumption $\mathbb{E}\left[\ell_{1,1}\Vert\alpha\Vert_1\right]< 1$,  it is not hard to see $G_2(0)\le\dfrac{1+\text{Var}(\ell_{1,1})\Vert\alpha\Vert^2_1}{2\left(1-\mathbb{E}\left[\ell_{1,1}\Vert\alpha\Vert_1\right]\right)^3}$. 

Now, let's assume that $G_2(s)\le\dfrac{1+\text{Var}(\ell_{1,1})\Vert\alpha\Vert^2_1}{2\left(1-\mathbb{E}\left[\ell_{1,1}\Vert\alpha\Vert_1\right]\right)^3}$. Then we can compute,
\begin{align*}
    G_2(s+1)
    &=\mathbb{E}[\ell_{1,1}\sum^{s+1}_{i=1}\alpha(i)G_{2}(s+1-i)] + \dfrac{1}{2}+(G_1(s)-1)\\
    &\quad+\dfrac{1}{2}\mathbb{E}\left[\left(\ell_{1,1}\sum^{s+1}_{i=1}\alpha(i)G_1(s+1-i)\right)^2\right]\\
    &\le \mathbb{E}\left[\ell_{1,1}\sum^{s+1}_{i=1}\alpha(i)\dfrac{1+\text{Var}(\ell_{1,1})\Vert\alpha\Vert^2_1}{2\left(1-\mathbb{E}\left[\ell_{1,1}\Vert\alpha\Vert_1\right]\right)^3}\right] + \dfrac{1}{2}+\dfrac{\mathbb{E}\left[\ell_{1,1}\Vert\alpha\Vert_1\right]}{1-\mathbb{E}\left[\ell_{1,1}\Vert\alpha\Vert_1\right]}\\
    &\quad+\dfrac{1}{2}\mathbb{E}\left[\left(\ell_{1,1}\sum^{s}_{i=1}\alpha(i)\dfrac{1}{1-\mathbb{E}\left[\ell_{1,1}\Vert\alpha\Vert_1\right]}\right)^2\right]\\
    &\le \dfrac{1+\text{Var}(\ell_{1,1})\Vert\alpha\Vert^2_1}{2\left(1-\mathbb{E}\left[\ell_{1,1}\Vert\alpha\Vert_1\right]\right)^3}\mathbb{E}\left[\ell_{1,1}\Vert\alpha\Vert_1\right] + \dfrac{1}{2}+\dfrac{\mathbb{E}\left[\ell_{1,1}\Vert\alpha\Vert_1\right]}{1-\mathbb{E}\left[\ell_{1,1}\Vert\alpha\Vert_1\right]}\\
    &\quad+\dfrac{\mathbb{E}\left[\ell^2_{1,1}\right]\Vert\alpha\Vert^2_1}{2\left(1-\mathbb{E}\left[\ell_{1,1}\Vert\alpha\Vert_1\right]\right)^2}\\
    &\le\dfrac{1+\text{Var}(\ell_{1,1})\Vert\alpha\Vert^2_1}{2\left(1-\mathbb{E}\left[\ell_{1,1}\Vert\alpha\Vert_1\right]\right)^3}.
\end{align*}
Hence,  we  proved  that  for  every $s\in\mathbb{N}$, 
$ G_2(s)\le\dfrac{1+\text{Var}(\ell_{1,1})\Vert\alpha\Vert^2_1}{2\left(1-\mathbb{E}\left[\ell_{1,1}\Vert\alpha\Vert_1\right]\right)^3}$. And similarly, we can also show that $\Tilde{G}_2(s)\le \dfrac{\mathbb{E}\left[\ell^2_{1,1}\right]}{2\left(1-\mathbb{E}\left[\ell_{1,1}\Vert\alpha\Vert_1\right]\right)^3}$.
\end{proof}
\newpage
\bibliographystyle{alpha}
\bibliography{mybibfile}

\end{document}